\documentclass[11pt]{article} 

\usepackage[utf8]{inputenc} 

\usepackage{geometry} 
\usepackage{graphicx} 

\usepackage{booktabs} 
\usepackage{array} 
\usepackage{paralist} 
\usepackage{verbatim} 
\usepackage{subfig} 
\usepackage{amsfonts}
\usepackage{amsthm}
\usepackage{amsmath}
\usepackage{amssymb}

\def \lg {{\ell g}}

\usepackage{fancyhdr} 
\usepackage{sectsty}
\allsectionsfont{\sffamily\mdseries\upshape} 

\everymath{\displaystyle}

\theoremstyle{definition}
\newtheorem{theorem}{Theorem}[section]

\newtheorem{lemma}[theorem]{Lemma}

\newtheorem{definition}[theorem]{Definition}
\newtheorem{remark}[theorem]{Remark}

\begin{document}
\author{Douglas Ulrich \thanks{The author was partially supported by NSF Research Grant DMS-1308546.}
	 \\ Department of Mathematics\\University of Maryland, College Park}
	\title{The Number of Atomic Models of Uncountable Theories}
	\date{\today} 
	
	\maketitle

\begin{abstract}
We show there exists a complete theory in a language of size continuum possessing a unique atomic model which is not constructible. We also show it is consistent with $ZFC + \aleph_1 < 2^{\aleph_0}$ that there is a complete theory in a language of size $\aleph_1$ possessing a unique atomic model which is not constructible. Finally we show it is consistent with $ZFC + \aleph_1 < 2^{\aleph_0}$ that for every complete theory $T$ in a language of size $\aleph_1$, if $T$ has uncountable atomic models but no constructible models, then $T$ has $2^{\aleph_1}$ atomic models of size $\aleph_1$.
\end{abstract}

\section{Introduction}
\label{Intro}

There are several model-theoretic notions of ``smallness," namely: a model $M$ is \emph{atomic} if every tuple $\overline{a} \in M$ has its type isolated by a single formula; a model $M$ is \emph{prime} if for every $N \equiv M$, there is an elementary embedding of $M$ into $N$; and a model $M$ is \emph{constructible} if there is a sequence $M = (a_\alpha: \alpha < \alpha^*)$ such that each $tp(a_\alpha / \{a_\beta: \beta < \alpha\})$ is isolated by a single formula.

If we are just interested in the complete theories in countable languages, then then these notions all coincide, by an old theorem of Vaught \cite{vaught} (essentially):

\begin{theorem}\label{vaught}
For models of $T$ a countable complete theory, the notions ``countable atomic," ``prime" and ``constructible" coincide. Such a model exists if and only if the isolated types are dense in the Stone spaces $S^n(\emptyset)$ for all $n$; when they exist they are unique up to isomorphism.
\end{theorem}

When we ask about theories in uncountable languages, things get harder. We have the following examples:

\begin{itemize}
\item Laskowski and Shelah \cite{LaskShelah}: there is a complete theory $T$ in a language of size $\aleph_2$, such that the isolated types are dense in $S^n(\emptyset)$ for all $n$, but $T$ has no atomic models.
\item Knight \cite{knight}: there is a complete theory in a language of size $\aleph_1$, with atomic models but no prime models.
\item Folklore: there is a complete theory in a language of size continuum, with prime models but no atomic models. Namely $Th(2^\omega, \mathbf{f}, U_n)_{f \in 2^\omega, n \in \omega}$, where $\mathbf{f}(g) = f \oplus g$ mod 2, and $U_n(g)$ holds iff $g(n) = 1$.
\item Shelah \cite{ShelahEx}: there is a complete theory in a language of size $\aleph_1$, with models that are atomic but not prime, and with models that are prime but not constructible, and with a constructible model. Namely $Th(\omega^{\omega_1}, E_\alpha: \alpha < \omega_1)$, where $\eta E_\alpha \tau$ iff $\eta \restriction_\alpha = \tau \restriction_\alpha$.
\end{itemize}

The following, one of the few positive results, was proved by Ressayre, see for example \cite{modelTheory}:

\begin{theorem}\label{constructibleModels}
Let $T$ be a complete theory in an arbitrary language. If $T$ has a constructible model $M$, then $M$ is unique up to isomorphism; it is furthermore prime and atomic. Also, the construction sequence for $M$ can be chosen of order type $|T|$.
\end{theorem}

And the following was proved independently by Knight \cite{knight}, Kueker \cite{KuekerIso} and Shelah \cite{ShelahIso}:

\begin{theorem}\label{atomicModelsAtAleph1}
Let $T$ be a complete theory in a language of size $\aleph_1$. Then $T$ has an atomic model if and only if the isolated types are dense in $S^n(\emptyset)$ for all $n$.
\end{theorem}

In this paper we are specifically interested in looking at the atomic models of $T$; we wonder when, for example, there exists a constructible model. Knight's example above shows that the answer is ``not always" but we would like to say more. In fact Knight's example has $2^{\aleph_1}$ models of size $\aleph_1$. We wonder if this is a necessary feature: that is, suppose $T$ is a complete theory in a language of size $\kappa$, with a unique atomic model $M$ of size $\leq \kappa$. Must $M$ be constructible?

By Vaught's Theorem~\ref{vaught}, for $\kappa = \aleph_0$ we know this to be true. We introduce the following examples to show it is false for $\kappa = 2^{\aleph_0}$.

\vspace{ 2mm}

\noindent \textbf{First Example: Theorem~\ref{continuumTheorem}.} There is a complete theory in a langauge of size continuum, with a unique atomic model, which is not prime. (Hence there are no prime models.) 

\vspace{2 mm}

\noindent \textbf{Second Example: Remark~\ref{ThirdExample}.} There is a complete theory in a language of size continuum, with a unique atomic model, which is furthermore prime, but which is not constructible. 

\vspace{2 mm}

Do we need continuum? This is only interesting if $\aleph_1 < 2^{\aleph_0}$; and with that assumption it turns out to be independent of $ZFC$. In particular we have the following theorems:

\vspace{2 mm}

\noindent \textbf{Third Example: Theorem~\ref{unifTheorem}.} IT is consistent with $ZFC + \aleph_1 < 2^{\aleph_0}$ that there is a complete theory in a language of size $\aleph_1$, with a unique atomic model, which is furthermore prime, but which is not constructible.

\begin{theorem}
\label{submain}
It is consistent with $ZFC + \aleph_1 < 2^{\aleph_0}$ that whenever $T$ is a complete theory in a language of size $\aleph_1$, if $T$ has atomic models but no constructible models, then $T$ has $2^{\aleph_1}$ atomic models of size $\aleph_1$.
\end{theorem}

The paper is organized as follows:

In Section~\ref{Background} we explain the various set-theoretic tools we use in the paper, and give sharper statements of the Third Example and of Theorem~\ref{submain}. The First Example is given in Section~\ref{ContinuumExample}, the Second and Third Examples are given in Section~\ref{UniformizationExample} and in Section~\ref{cons} we prove Theorem~\ref{submain}.

The author thanks Chris Laskowski for suggesting these problems, for many helpful discussions and for many helpful comments on the writing of this paper.
\section{Background, and Statement of Results} \label{Background}

We first review the set-theoretic notions required for the consistency proofs. \cite{diamond} serves as a general reference.

\vspace{2 mm}

\subsection{Ladder Systems} Let $\Lambda \subseteq \omega_1$ be the limit ordinals. Suppose $S \subset \Lambda$ is stationary. A \emph{ladder system} $(L_\alpha: \alpha \in S)$ is a sequence of subsets of $\omega_1$ such that for each $\alpha \in S$, $L_\alpha \subset \alpha$ is cofinal and of order type $\omega$. $(L_\alpha: \alpha \in S)$ has the \emph{uniformization property} if for every sequence $(f_\alpha: \alpha \in S)$ of functions $f_\alpha: L_\alpha \to 2$, there is some $f: \omega_1 \to 2$ such that for all $\alpha \in S$, $\{b \in L_\alpha: f_\alpha(\beta) \not= f(\beta)\}$ is finite.

\vspace{2 mm}

We have the following, proven by Devlin-Shelah \cite{diamond2}:

\begin{theorem}
Martin's Axiom, together with $\aleph_1 < 2^{\aleph_0}$, implies that every ladder system on $\Lambda$ has the uniformization property (and hence that every ladder system on any stationary $S$ has the uniformization property.)

\end{theorem}

In particular $ZFC + \aleph_1 < 2^{\aleph_0} + $ ``every ladder system on $\Lambda$ has the uniformiation property" is equiconsistent with $ZFC$.

The uniformization property was originally introduced to analyze Whitehead groups. Namely, Shelah showed in \cite{WH2} that there is a non-free Whitehead group of size $\aleph_1$ if and only if for some stationary $S \subset \Lambda$, some ladder system on $S$ has the uniformization property.

We sharpen the Third Example as follows:

\vspace{2 mm}

\noindent \textbf{Third Example, Sharp Version.} Suppose there is some stationary $S \subseteq \omega_1$ that admits a ladder system. Then there is a theory $T$ in a language of size $\aleph_1$ such that $T$ has a unique atomic model, which is furthermore prime, but which is not constructible.
\vspace{2 mm}

\subsection{The Weak Diamond Principle}  

If $S \subseteq \omega_1$ is stationary, then $\Phi(S)$ (``weak diamond on $S$") is the combinatorial guessing-principle which states that for every $F: 2^{<\omega_1} \to 2$, there is some $g: S \to 2$ such that for every $f: \omega_1 \to 2$, the set $\{\alpha \in S: F(f \restriction_\alpha) = g(\alpha)\}$ is stationary. So the smaller $S$ is, the stronger $\Phi(S)$ is; $\Phi(\omega_1)$ is equivalent to $2^{\aleph_0} < 2^{\aleph_1}$. 

\begin{definition}
Let $\Phi^*$ abbreviate: for all stationary $S \subset \omega_1$, $\Phi(S)$ holds.
\end{definition}

It is easy to show that, for example, $\Phi(S)$ holds if and only if for every $F: (2 \times 2 \times \omega_1)^{<\omega_1} \to 2$, there is some $g: S \to 2$ such that for every $f_0, f_1: \omega_1 \to 2$ and for every $h: \omega_1 \to \omega_1$, the set $\{\alpha \in S: F(f_0 \restriction_\alpha, f_1 \restriction_\alpha, h \restriction_\alpha) = g(\alpha)\}$ is stationary.

These principles were introduced by Devlin and Shelah \cite{diamond2}, where they proved the following theorems:

\begin{theorem} \label{split}
\begin{enumerate}
\item $\Phi(\omega_1)$ is equivalent to $2^{\aleph_0} < 2^{\aleph_1}$. 
\item  Suppose $\Phi(S)$ holds. Then we can write $S$ as the disjoint union of stationary sets $(S_\alpha: \alpha < \omega_1)$ such that $\Phi(S_\alpha)$ holds for each $\alpha$.
\item Suppose $S \subseteq \Lambda$ is stationary. If $\Phi(S)$ holds then no ladder system on $S$ has the uniformization property.
\end{enumerate}
\end{theorem}

In view of the first item, $\Phi^*$ is a strengthening of $2^{\aleph_0} < 2^{\aleph_1}$.

\subsection{The Covering Number} Let $\mathsf{Cov}(\mathcal{K})$ be the covering number of the $\sigma$-ideal of meager sets: i.e. the least $\kappa$ such that $2^{\omega}$ is the union of $\kappa$-many closed nowhere dense sets. This is a well-understood cardinal invariant of the continuum. In particular $\omega < \mathsf{Cov}(\mathcal{K}) \leq 2^{\aleph_0}$, and if Martin's Axiom holds then $\mathsf{Cov}(\mathcal{K}) = 2^{\aleph_0}$.

So $\Phi^* \land \mathsf{Cov}(\mathcal{K}) \geq \aleph_2$ says that $\aleph_1 < 2^{\aleph_0} < 2^{\aleph_1}$ in a strong way. This assertion is consistent: let $\mathbb{P}$ be the forcing notion $\mathbb{P}_0 \times \mathbb{P}_1$ where $\mathbb{P}_0 = \mbox{Fn}(\omega_2, 2, \omega)$ and $\mathbb{P}_1 = \mbox{Fn}(\omega_3, 2, \omega_1)$. (Here $\mbox{Fn}(X, Y, \kappa)$ is the set of all partial functions $f$ with domain $\subseteq X$ and range $\subseteq Y$, and with $|f| < \kappa$.) $\mathbb{P}$ is the standard forcing notion for arranging $2^{\aleph_0} = \aleph_2$, $2^{\aleph_1} = \aleph_3$, starting from GCH. Then we have:

\begin{theorem}
\label{consistency}
Suppose $\mathbb{V} \models GCH$ and $G$ is $\mathbb{P}$-generic over $\mathbb{V}$. Then $\mathbb{V}[G] \models \Phi^* \land \mathsf{Cov}(\mathcal{K}) \geq \aleph_2$.
\end{theorem}
\begin{proof}
It is shown in \cite{forcing} (Theorem 2.11 from the appendix) that $\mathbb{V}[G] \models \Phi^*$. 

Note that these forcing notions all preserve cardinals, so we can refer to $\omega_1$, etc., without ambiguity.

Let $G_1$ be $\mathbb{P}_1$-generic over $\mathbb{V}$. Working in $\mathbb{V}[G_1]$, we show that if $G_0$ is $\mathbb{P}_0$-generic over $\mathbb{V}[G_1]$ then $\mathbb{V}[G_1][G_0] \models \mathsf{Cov}(\mathcal{K}) \geq \aleph_2$. 

Indeed, suppose $(C_\alpha: \alpha < \omega_1)$ is a sequence in $\mathbb{V}[G_1][G_0]$ of closed nowhere dense subsets of $(\omega^\omega)^{\mathbb{V}[G_1][G_0]}$. Let $x_\alpha \in (\omega^\omega)^{\mathbb{V}[G_1][G_0]}$ encode $C_\alpha$.

In $\mathbb{V}[G_1]$, write $\omega_2 = I \cup J$ where $I, J$ are disjoint, $|I| \leq \aleph_1$, and such that setting $H_0 = G_0 \restriction_I$, $H_1 = G_0 \restriction_J$, we have that $(x_\alpha: \alpha < \omega_1) \subset \mathbb{V}[G_1][H_0]$. Choose a real $x \in \mathbb{V}[G_1][G_0] = \mathbb{V}[G_1][H_0][H_1]$ such that $x$ is Cohen over $\mathbb{V}[G_1][H_0]$; then $x \not \in C_\alpha$ for all $\alpha < \omega_1$, showing that $\bigcup_{\alpha < \omega_1} C_\alpha \not= (\omega^\omega)^{\mathbb{V}[G_0][G_1]}$.
\end{proof}

\subsection{The Main Theorem}

For $T$ a complete theory in a countable language, the question of the number of atomic models of $T$ of size $\aleph_1$ has been closely investigated. First of all, such models exist if and only if $T$ has a (unique) countable atomic model, which furthermore has a proper atomic extension. Assuming this, let $\mathbf{K}_T$ be the class of atomic models of $T$.

Now say that $\mathbf{K}_T$ is $\omega$-\emph{stable} if $S^n_{at}(M)$ is countable for all $n$, where $M$ is some countable atomic model of $T$, and $S^n_{at}(M)$ is the set of all $n$-types $p(\overline{x}) \in S^n(M)$ such that $M \overline{a}$ is atomic whenever $\overline{a}$ realizes $p(\overline{x})$.

Then we have the following theorems of Shelah \cite{ShelahClass1} \cite{ShelahClass2} (or see \cite{cat} for an exposition):

\begin{theorem} \label{Shelah1}
Suppose $2^{\aleph_0} < 2^{\aleph_1}$, and $\mathbf{K}_T$ is not $\omega$-stable. Then $T$ has $2^{\aleph_1}$ nonisomorphic models of size $\aleph_1$.
\end{theorem}

It is not known if the assumption $2^{\aleph_0} < 2^{\aleph_1}$ is necessary here. On the other hand, if $\mathbf{K}_T$ is $\omega$-stable, then we have a strong enough structure theory to determine e.g. when $\mathbf{K}_T$ is $\aleph_1$-categorical.

Now, our main theorem (Theorem~\ref{main} below) will be essentially a generalization of Theorem~\ref{Shelah1}, and will follow the same general proof outline, which we now describe.

Namely, the proof of Theorem~\ref{Shelah1} splits into cases depending on whether $\mathbf{K}_T$ has the amalgation property at $\aleph_0$. Here, an amalgamation problem at $\aleph_0$ (for $\mathbf{K}_T$) is a triple $(M_0, M_1, M_2)$ where each $M_i$ is a countable atomic model of $T$, and $M_0 \preceq M_i$ for $i = 1, 2$. A solution to the amalgamation problem is a triple $(M_3, f_1, f_2)$ where $M_3$ is a countable atomic model of $T$, and $f_i: M_i \preceq M_3$, and $f_1 \restriction_{M_0} = f_2 \restriction_{M_0}$. $\mathbf{K}_T$ has the amalgamation property at $\aleph_0$ if every amalgamation problem at $\aleph_0$ has a solution.

So to prove Theorem~\ref{Shelah1}, we first consider the case where $\mathbf{K}_T$ fails the amalgamation property at $\aleph_0$, and then the case where $\mathbf{K}_T$ has the amalgamation property at $\aleph_0$ but is not $\omega$-stable. 

But it is worth noting that we have the following Corollary 19.14 from \cite{cat}:

\begin{theorem}
If $\mathbf{K}_T$ is $\omega$-stable then $\mathbf{K}_T$ has the amalgamation property at $\aleph_0$.
\end{theorem}

We will also want the following strengthening (an easy consequence of Corollary 24.4 from \cite{cat}). To state it conveniently we work in a monster model $\mathfrak{C}$ of $T$. Say that a set $A \subset \mathfrak{C}$ is atomic if every finite tuple from $A$ realizes an isolated type.

\begin{theorem}\label{countLangAmalg}
Suppose $\mathbf{K}_T$ is $\omega$-stable, and $(A_0, A_1, A_2)$ is a triple of countable atomic sets with $A_0 \subseteq A_i$ for $i = 1, 2$. Suppose $S_{at}(A_0)$ is countable. Then $(A_0, A_1, A_2)$ can be amalgamated by some countable atomic set $A_3$.
\end{theorem}

\vspace{5 mm}

Let $T$ be a complete theory in a language of size $\aleph_1$, and let $\mathbf{K}_T$ be its class of atomic models. In Section~\ref{cons} we define the notion ``$\mathbf{K}_T$ is club totally transcendental," generalizing the definition of $\omega$-stability for $\mathcal{L}$ countable. We then prove our main theorem, a sharpening of Theorem~\ref{submain}:

\begin{theorem}\label{main}
Suppose $\Phi^*$ holds, and $\mathsf{Cov}(\mathcal{K}) \geq \aleph_2$. Suppose $T$ is a complete theory in a language of size $\aleph_1$ with atomic models, and $\mathbf{K}_T$ is not club totally transcendental. Then $T$ has $2^{\aleph_1}$ atomic models of size $\aleph_1$.
\end{theorem}

The hypotheses can be understood as follows: the First and Second Examples require CH (to matter at $\aleph_1$), which $\mathsf{Cov}(\mathcal{K}) \geq \aleph_2$ prevents; and the Third Example requires the existence of ladder systems with the uniformization property, which $\Phi^*$ prevents.

The proof of Theorem~\ref{main} follows the same outline as that of Theorem~\ref{Shelah1}. Namely, we will say what it means for $\mathbf{K}_T$ to have the club amalgamation property, then split into two cases, depending on whether $\mathbf{K}_T$ fails the club amalgamation property, or else $\mathbf{K}_T$ has the club amalgamation property but it not club totally transcendental. 

As in the countable case we can actually show that club totally transcendental implies the club amalgamation property; this is discussed in Section~\ref{amalgamation}. However this is not technically needed for the proof.

Finally, one obtains Theorem~\ref{submain} quickly, since if $T$ has no constructible models then $\mathbf{K}_T$ is not club totally transcendental; see Section~\ref{ttAndCons}.
\section{Unique Atomic Model that is not Prime}
\label{ContinuumExample}

In this section I construct the First Example: namely an atomic model $\mathfrak{A} \models T$, in a language of size continuum, which has a unique atomic model that is not prime.

Given $\eta \in 2^{<\omega_1}$, let $\lg(\eta)$ be its length, i.e. its domain.

Let $\mathcal{L} = (U_\alpha,  \pi_{\alpha \beta}, \boldsymbol{\eta} : \beta \leq \alpha < \omega_1, \eta \in 2^{<\omega_1})$ where each $U_\alpha$ is a unary relation symbol, each $\boldsymbol{\eta}: U_{\lg(\eta)} \to U_{\lg(\eta)}$ is a unary function symbol, and each $\pi_{\alpha \beta}: U_\alpha \to U_\beta$ is a unary function symbol. (Formally, since we are using single-sorted logic, each of these function symbols will be total, but we let their values be trivial outside their domain.)

We turn $2^{<\omega_1}$ into an $\mathcal{L}$-structure $\mathfrak{A} = (2^{<\omega_1}, U_\alpha, \pi_{\alpha \beta}, \boldsymbol{\eta}: \beta \leq \alpha < \omega_1, \eta \in 2^{<\omega_1})$ as follows:

\begin{itemize}
\item Interpret $U_\alpha$ as $2^\alpha$, i.e. all $\eta \in 2^{<\omega_1}$ with $\lg(\eta) = \alpha$;
\item Given $\tau \in 2^\alpha$ and $\beta \leq \alpha$, interpret $\pi_{\alpha \beta}(\tau)$ as $\tau \restriction_\beta$;
\item Given $\eta \in 2^\alpha$ and $\tau \in U_\alpha$ interpret $\boldsymbol{\eta}(\tau)$ as $\eta \oplus \tau$, where the addition is pointwise mod 2.
\end{itemize}

Let $T$ be the complete theory of $\mathfrak{A}$.

\begin{theorem}\label{continuumTheorem}
$\mathfrak{A}$ is the unique atomic model of $T$, and it is not prime.
\end{theorem}

The proof goes as follows. First we establish that $\mathfrak{A}$ is the unique atomic model of $T$. Then we give an axiomatization of $T$, and use it to exhibit a model $\mathfrak{B}$ of $T$ into which $\mathfrak{A}$ does not embed; in fact $\mathfrak{B}$ will omit $tp_{\mathfrak{A}}(\overline{0}_\alpha: \alpha < \omega_1)$, where $\overline{0}_\alpha \in 2^\alpha$ is the zero sequence.

\begin{lemma} We write down some straightforward observations:
\begin{itemize}
\item Given $\eta, \tau \in 2^{\alpha}$, $\boldsymbol{\eta} \boldsymbol{\tau} = \boldsymbol{\tau} \boldsymbol{\eta}$;
\item Given $\eta \in 2^{\alpha}$ and $\beta \leq \alpha$, $\pi_{\alpha \beta} \boldsymbol{\eta} = \,\boldsymbol{\eta \! \restriction_{\beta}} \pi_{\alpha \beta}$;
\item Given $\gamma \leq \beta \leq \alpha$, $\pi_{\beta \gamma} \pi_{\alpha \beta} = \pi_{\alpha \gamma}$, and $\pi_{\alpha \alpha}$ is the identity on $2^\alpha$;
\item Given $\nu \in 2^{\omega_1}$, the map $f_\nu: \mathfrak{A} \to \mathfrak{A}$ defined by $f_\nu(\eta) = \, \boldsymbol{\nu \! \restriction_{\lg(\eta)}}(\eta)$ is an automorphism of $\mathfrak{A}$;
\item For all $\eta \in 2^\alpha, \tau \in 2^\beta$, if $\alpha \geq \beta$ then $\tau$ is in the definable closure $\eta$, namely $\tau = \boldsymbol{\tau} \pi_{\alpha \beta} \boldsymbol{\eta}(\eta)$.
\item In particular, for all $\eta, \tau$, either $\eta$ is in the definable closure of $\tau$ or vice versa.
\end{itemize}
\end{lemma}

\begin{lemma}
\label{types1}
Suppose $\overline{\eta} = \eta_0 \ldots \eta_{n-1}$ is a finite sequence from $2^{<\omega_1}$. Write $\alpha_i = \lg(\eta_i)$; we can suppose $\alpha_0 \geq \alpha_i$ for all $i < n$. Then the formula 

\[\phi_{\overline{\eta}}(x_0 \ldots x_{n-1}) := U_{\alpha_0}(x_0) \,\,\, \land \,\,\, \bigwedge_{i < n} x_i = \boldsymbol{\eta_i}\pi_{\alpha_0\alpha_i}  \boldsymbol{\eta_{0}} (x_0)\]

isolates $tp_{\mathfrak{A}}(\overline{\eta})$. In particular $\mathfrak{A}$ is an atomic model of $T$.
\end{lemma}

\begin{proof}
It is clear that $\mathfrak{A} \models \phi_{\overline{\eta}}(\overline{\eta})$. Conversely, suppose $\mathfrak{A} \models \phi_{\overline{\eta}}(\tau_0, \ldots, \tau_{n-1})$. Let $\nu := (\tau_0 \oplus \eta_0) \,^\frown \overline{0} \in 2^{\omega_1}$. Then $f_{\nu}: \mathfrak{A} \cong \mathfrak{A}$ (defined above) is an automorphism of $\mathfrak{A}$ taking $\overline{\eta}$ to $\overline{\tau}$, so they have the same type.
\end{proof}

\begin{lemma}
\label{continuumUniqueness}
$\mathfrak{A}$ is the unique atomic model of $T$.
\end{lemma}
\begin{proof}
Suppose $\mathfrak{B} \models T$ is atomic; say $\mathfrak{B} = (B, U^*_\alpha, \pi^*_{\alpha\beta}, \boldsymbol{\eta}^* : \eta \in 2^{<\omega_1}, \beta \leq \alpha < \omega_1)$.

Note that Lemma~\ref{types1} characterizes all the complete isolated types of $T$. In particular $B = \bigcup_{\alpha <\omega_1} U^*_\alpha$.

We define by induction on $\alpha < \omega_1$ an element $b_\alpha \in U^*_\alpha$ such that for all $\beta \leq \alpha < \omega_1$, $b_\beta = \pi^*_{\alpha \beta}(b_\alpha)$.

There is a unique element of $U^*_0$, so we let that element be $b_0$.

Suppose we have defined $b_\alpha$. Then let $b_{\alpha+1}$ be either of the two elements in $U^*_{\alpha+1}$ that restrict to $b_\alpha$.

Finally, suppose $\alpha < \omega_1$ is a limit, and we have defined $b_\beta$ for all $\beta < \omega_1$. Let $b \in U^*_\alpha$ be arbitrary. For each $\beta < \alpha$, let $\eta_\beta \in 2^\beta$ be the unique function with $b_\beta = \boldsymbol{\eta}_{\boldsymbol{\beta}}\pi_{\alpha \beta}(b)$. Then $\eta_\beta \subseteq \eta_\gamma$ for $\beta \leq \gamma < \alpha$. Define $\eta = \bigcup_{\beta < \alpha} \eta_\beta$, and define $b_\alpha = \boldsymbol{\eta}(b)$. This works, clearly.

So we have $(b_\alpha: \alpha < \omega_1)$ as desired. For each $\alpha < \omega_1$, let $\overline{0}_\alpha \in \mathfrak{A}$ be the zero sequence of length $\alpha$. Then $f: \overline{0}_\alpha \mapsto b_\alpha$ is a partial elementary map from $\mathfrak{A}$ into $\mathfrak{B}$. So $f$ extends to a partial elementary map $g$ from the definable closure of $\{\overline{0}_\alpha: \alpha < \omega_1\}$ in $\mathfrak{A}$ to the definable closure of $\{b_\alpha: \alpha < \omega_1\}$. 

But note that the definable closure of each $\overline{0}_\alpha$ contains all of $U_\alpha$, and the definable closure of each $b_\alpha$ contains all of $U_\alpha^*$. Hence $g: \mathfrak{A} \cong \mathfrak{B}$.
\end{proof}

Now we provide an axiomatization of $T$.

\begin{definition}
Let $T_0 $ consist of the consequences of the following axioms.

\begin{enumerate}[(I)]
\item Suppose $\phi(x)$ is a quantifier-free formula of $\mathcal{L}$ with only the variable $x$ free. Suppose $\mathfrak{A} \models \forall x \phi(x)$. Then $``\forall x \phi(x)"$ is an axiom.
\item $``\exists! x : U_0(x)".$
\item For all $\alpha$, $``\forall x : U_\alpha(x) \rightarrow \exists^{=2}y : (U_{\alpha+1}(y) \land \pi_{\alpha+1\,\alpha}(y) = x)"$.
\item For all $\alpha < \beta$, ``$\forall x: U_\alpha(x) \rightarrow \exists y: (U_\beta(y) \land \pi_{\beta \alpha}(y) = x)$".
\end{enumerate}
\end{definition}

Obviously $\mathfrak{A} \models T_0$.

\begin{lemma} $T_0 = T$, i.e. $T_0$ is complete.
\end{lemma}
\begin{proof} (Sketch.) 
It suffices to show that for sufficiently rich finite fragments $\mathcal{L}' \subseteq \mathcal{L}$, $T_0 \restriction_{\mathcal{L}'}$ is $\aleph_0$-categorical.

Temporarily define a template to be a sequence $\overline{G} = (G_\alpha: \alpha \in X)$ where:

\begin{itemize}
\item $X \subseteq \omega_1$ is finite and closed under immediate predecessors, and $0 \in X$;
\item Each $G_\alpha$ is a finite subgroup of $(2^\alpha, \oplus)$, containing the set of all $\eta \in 2^\alpha$ which are zero outside of $X$;
\item For $\beta \leq \alpha$ both in $X$, $G_\alpha \! \restriction_\beta = G_\beta$.
\end{itemize}

Given a template $\overline{G} = (G_\alpha: \alpha \in X)$, let $\mathcal{L}_{\overline{G}} \subseteq \mathcal{L}$ be defined as follows: $U_\alpha \in \mathcal{L}_{\overline{G}}$ iff $\alpha \in X$; $\pi_{\alpha \beta} \in \mathcal{L}_{\overline{G}}$ iff $\alpha, \beta \in X$; and $\boldsymbol{\eta} \in \mathcal{L}_{\overline{G}}$ iff $\lg(\eta) \in X$ and $\eta \in G_{\lg(\eta)}$. Let $T_{\overline{G}} = T_0 \restriction_{\mathcal{L}_{\overline{G}}}$.

Then it is easy to see that each $T_{\overline{G}}$ is $\aleph_0$-categorical; note for example that $T_{\overline{G}}$ proves there are infinitely many unsorted elements (i.e. elements that are not in any $U_\alpha$ for $\alpha \in X$) and that these elements are absolutely indiscernible over the rest of the model.
\end{proof}

The following lemma concludes the proof of Theorem~\ref{continuumTheorem}.

\begin{lemma}
$\mathfrak{A}$ is not a prime model of $T$.
\end{lemma}
\begin{proof}
We define a model $\mathfrak{B} = (B, U^*_\alpha, \pi^*_{\alpha\beta}, \boldsymbol{\eta}^*: \eta \in 2^{<\omega_1}, \beta \leq \alpha < \omega_1) \models T$ into which $\mathfrak{A}$ does not embed.

\begin{itemize}
\item Let $B$ be the set of all pairs $(\tau, s)$ where:

\begin{itemize}
\item $\tau \in 2^{<\omega_1}$;
\item $s \in \omega_1^{<\omega}$ is a finite, strictly increasing sequence of ordinals, with $|s| \geq 2$;
\item $s(0) = 0$, $s(1) = \omega$, and for all $n \geq 1$, $s(n)$ is a limit ordinal;
\item $s(|s|-2) \leq \lg(\tau) < s(|s|-1)$.
\end{itemize}
\item Suppose $(\tau, s) \in B$. Then let $U^*_\alpha(\tau, s)$ hold iff $\tau \in 2^\alpha$.
\item Suppose $(\tau, s) \in U^*_\alpha$ and $\eta \in 2^\alpha$. Then let $\boldsymbol{\eta}^*(\tau, s) = (\eta \oplus \tau, s)$.
\item Suppose $(\tau, s) \in U^*_\alpha$ and $\beta \leq \alpha$. Let $n$ be such that $s(n-2) \leq \beta < s(n-1)$. Let $\pi^*_{\alpha\beta}(\tau, s) = (\tau \restriction_\beta, s \restriction_n)$.

\end{itemize}

It is routine to check that $\mathfrak{B}$ is a model of Axiom Schemas II-IV. To check Axiom Schema I: suppose $\mathfrak{B} \models \exists x \phi(x)$, where $\phi(x)$ is a quantifier-free $\mathcal{L}$-formula. Say $\mathfrak{B} \models \phi(\eta, s)$. Let $A_0$ be the definable closure of $\eta$ in $\mathfrak{A}$ (i.e., all $\tau \in \mathfrak{A}$ with $\lg(\tau) \leq \lg(\eta)$) and let $B_0$ be the definable closure of $(\eta, s)$ in $\mathfrak{B}$ (i.e., all $(\tau, t) \in \mathfrak{B}$ with $\lg(\tau) \leq \lg(\eta)$ and $t \subseteq s$). Then the map $\Phi: B_0 \to A_0$ taking $(\tau, t)$ to $\tau$ is a partial isomorphism from $B_0$ onto $A_0$. Hence $\mathfrak{A} \models \phi(\eta)$, so $\mathfrak{A} \models \exists x \phi(x)$.

So $\mathfrak{B} \models T$. Suppose towards a contradiction that $f: \mathfrak{A} \to \mathfrak{B}$ were an elementary embedding. Let $\overline{0}_\alpha$ be the zero sequence of length $\alpha$ in $\mathfrak{A}$, for each $\alpha < \omega_1$; and let $(\eta_\alpha, s_\alpha) = f(\overline{0}_\alpha)$. Then we have for all $\alpha < \beta$, $\pi^*_{\beta \alpha}(\eta_\beta, s_\beta) = (\eta_\alpha, s_\alpha)$. In particular, for all $\alpha < \beta$, $s_\alpha \subseteq s_\beta$. 

Hence $(s_\alpha : \alpha < \omega_1)$ eventually stabilizes; say $s_\alpha = s_\beta = s$ for all $\alpha, \beta \geq \alpha_0$. Let $\alpha_1 = \max(s(|s|-1), \alpha_0)$. Then $\lg(\eta_{\alpha_1}) \geq s_\alpha(|s_\alpha|-1)$, contradicting the definition of $B$.
\end{proof}

\section{Unique Atomic Models that are Prime but not Constructible}
\label{UniformizationExample}

In this section, I show the following:

\begin{theorem}
\label{unifTheorem}
Third Example: Suppose for some stationary $S \subset \Lambda$, some ladder system $(L_\alpha: \alpha \in S)$ has the uniformization property. Then from this ladder system we can define a theory $T$ in a language $\mathcal{L}$ of size $\aleph_1$, such that $T$ has a unique atomic model, which is additionally prime, yet which is not constructible.
\end{theorem}

A small tweak (see Remark~\ref{ThirdExample} below) gives the Second Example.

The idea is to make an example similar to the first example, except we replace the tree $(2^{<\omega_1}, <)$ with a much smaller tree, in fact a tree of height $\omega+1$. (In neither example is $<$ itself part of the language.)

Fix a stationary $S \subset \Lambda$ and a ladder system $(L_\alpha: \alpha \in S)$ with the uniformization property. Let $\nu_\alpha: \omega \to L_\alpha$ be the strictly increasing enumeration.

Let $J_0$ be the set of all strictly increasing functions $\eta_0: \alpha \to \omega_1$, where $\alpha \leq \omega$, and if $\alpha = \omega$ then $\eta_0= \nu_\beta$ for some $\beta \in S$. So $J_0$ is a tree of height $\omega+1$ under $\subset$.

Let $J_1 = \{\eta_1 \in 2^{\leq \omega}: \eta_1 \mbox{ has finite support}\}$. $J_1$ is also tree of height $\omega+1$, under initial segment $\subset$.

Let $J = J_0 \otimes J_1$ be the product tree of pairs $\eta = (\eta_0, \eta_1)$, where $\eta_0 \in S_0$ and $\eta_1 \in S_1$ and $|\eta_0| = |\eta_1|$; we say $\eta \leq \tau$ if $\eta_i \leq \tau_i$ for each $i < 2$.

Given $\eta \in J$, we view $\eta$ as a sequence with domain $\alpha \leq \omega$, and write $\lg(\eta)$, $\eta \restriction_n$, etc. accordingly. If $\eta, \tau, \ldots \in J$ then always $\eta = (\eta_0, \eta_1)$, $\tau = (\tau_0, \tau_1), \ldots$.

Let $\mathcal{L}$ be the language $\{U_{\eta_0}, \pi_\alpha, \boldsymbol{\eta}: \eta_0 \in J_0, \eta \in J, \alpha \leq \omega\}$, where each $U_{\eta_0}$ is a unary predicate, and each $\pi_\alpha, \boldsymbol{\eta}$ are unary function symbols. $\boldsymbol{\eta}$ will be a map $U_{\eta_0} \to U_{\eta_0}$.

We turn $J$ into a $\mathcal{L}$-structure $\mathfrak{A}$ as follows. Let $U_{\eta_0} = \{\tau \in J: \tau_0 = \eta_0\}$. Given $\eta \in J$ and $\alpha \leq \omega$, let $\pi_\alpha(\eta) = \eta \restriction_{\alpha}$ (so $\pi_\omega$ is the identity map). Finally, given $\eta, \tau \in J$ with $\eta_0 = \tau_0$, define $\boldsymbol{\eta} \tau = (\eta_0, \eta_1 \oplus \tau_1 \bmod 2)$.

Let $T$ be the complete theory of $\mathfrak{A}$. The claim is that this works.

Given $\eta_0, \eta_1 \in J_0$, then let $d(\eta_0, \eta_1)$ be the greatest $\alpha \leq \omega$ such that $\alpha \leq \lg(\eta_0)$ and $\alpha \leq \lg(\eta_1)$ and $\eta_0 \restriction_\alpha = \eta_1 \restriction_\alpha$.
\begin{lemma}
\label{unifTypes}

Let $\overline{\eta} = (\eta^0, \eta^1, \ldots, \eta^{n-1})$ be a tuple from $J$. For each $i, j$ let $\alpha_{ij} = d(\eta^i_0, \eta^j_0)$. Let $\phi_{\overline{\eta}}(x_0, \ldots, x_{n-1})$ be the following formula:

$$\bigwedge_{i < n} U_{\eta^i_0}(x_i) \,\, \land \,\, \bigwedge_{i \not= j < n} \pi_{\alpha_{ij}} (\boldsymbol{\eta^i} x_i) = \pi_{\alpha_{ij}} (\boldsymbol{\eta^j} x_j).$$

Then $\mathfrak{A} \models \phi_{\overline{\eta}}(\overline{\eta})$ and moreover $\phi_{\overline{\eta}}(\overline{x})$ is complete.
\end{lemma}
\begin{proof}
It is clear that $\mathfrak{A} \models \phi_{\overline{\eta}}(\overline{\eta})$. We show that the formula is complete by defining, for each pair $\overline{\eta}, \overline{\tau}$ with $\mathfrak{A} \models \phi_{\overline{\eta}}(\overline{\tau})$, an automorphism $\sigma_{\overline{\eta}, \overline{\tau}} : \mathfrak{A} \cong \mathfrak{A}$ taking $\overline{\eta}$ to $\overline{\tau}$. We do this inductively on $n = |\overline{\eta}|$.

For $n = 0$ define $\sigma_{\emptyset, \emptyset} = \mbox{id}_{\mathfrak{A}}$.

Suppose we have defined $\sigma_{\overline{\eta}, \overline{\tau}}$ for all $|\overline{\eta}|, |\overline{\tau}| \leq n$. Let $\overline{\eta} = (\eta^0, \ldots, \eta^n)$ be given, and suppose $\mathfrak{A} \models \phi_{\overline{\eta}}(\overline{\tau})$. We can suppose, by applying $\sigma_{(\eta^0, \ldots, \eta^{n-1}), (\tau^0, \ldots, \tau^{n-1})}^{-1}$ to $\overline{\tau}$, that $\eta_i = \tau_i$ for each $i < n$. So we want to find some $\sigma: \mathfrak{A} \cong \mathfrak{A}$ such that $\sigma(\eta^i) = \eta^i$ for each $i < n$, and $\sigma(\eta^n) = \tau^n$.

Let $\alpha_{ij}$ be as in the definition of $\phi_{\overline{\eta}}(\overline{x})$: $\alpha_{ij} = d(\eta^i_0, \eta^j_0)$.

We know that $\eta^n_0 = \tau^n_0$. If $\eta^n_1 = \tau^n_1$ then we are done, so suppose $\eta^n_1 \not= \tau^n_1$. Let $m < \omega$ be the least value at which they differ (so $m$ is greatest such that $\eta^n_1 \restriction_m = \tau^n_1 \restriction_m$).

Then for each $i < n$, $\alpha_{i \, n} \leq m$, since by the $(i, n)$ clause of $\phi_{\overline{\eta}}$ we have that \\$(\eta^n_1 \oplus \eta^n_1) \restriction_{\alpha_{i\,n}} = (\eta^i_1 \oplus \eta^i_1) \restriction_{\alpha_{i \, n}} = (\eta^n_1 \oplus \tau^n_1) \restriction_{\alpha_{i\,n}}$.

Define $\sigma$ as follows: suppose $\eta \in \mathfrak{A}$. Then $\sigma(\eta) = \tau$ where $\tau_0 = \eta_0$, where $\tau_1(k) = \eta_1(k) +\eta^n_1(k) + \tau^n_1(k)  \bmod 2$ for $k < d(\eta_0, \eta^n_0)$, and $\tau_1(k) = \eta_1(k)$ for $k \geq d(\eta_0, \eta^n_0)$. Then it is simple to check that $\sigma$ has the desired properties.
\end{proof}

From this it is clear that the algebraic closure of the emptyset $acl(\emptyset)^\mathfrak{A}$ is just $\{\eta \in J: \lg(\eta) < \omega\}$. Denote this set as $X$.

We define an auxilary $\mathcal{L}$-structure $\mathfrak{M} = (\mathbf{J}, U_{\eta_0}, \pi_\alpha, \boldsymbol{\eta}: \eta \in J, \alpha \leq \omega)$ similarly to $\mathfrak{A}$: namely $\mathbf{J} = J_0 \otimes 2^{\leq \omega}$, with the natural operations. So $\mathfrak{A}$ is a substructure of $\mathfrak{M}$.

In fact $\mathfrak{A} \preceq \mathfrak{M}$ but we won't need this.

Given a sequence $\mathcal{F} = (f_\alpha: \alpha \in S)$, where each $f_\alpha \in 2^\omega$, define $J_{\mathcal{F}}$ to be be $X$, together with all pairs $(\eta_0, \eta_1) \in \mathbf{J}$ where $\eta_0 = \nu_\alpha$ is the canonical enumeration of $L_\alpha $ (defined at the beginning of the section) and where $\eta_1$ differs only finitely often from $f_\alpha$. Define $\mathfrak{A}_{\mathcal{F}}$ to be the substructure of $\mathfrak{M}$ with domain $J_{\mathcal{F}}$.

Note that $\mathfrak{A} = \mathfrak{A}_{(\overline{0}: \alpha \in S)}$.
\begin{lemma}
Each $\mathfrak{A}_{\mathcal{F}} \cong \mathfrak{A}$.
\end{lemma}
\begin{proof}

Fix $\mathcal{F} = (f_\alpha: \alpha \in S)$. Define $s_\alpha: L_\alpha \to 2$ by $s_\alpha(\nu_\alpha(n)) = f_\alpha(n)$. By the uniformization property we can choose some $s: \omega_1 \to 2$ such that $s$ differs from each $s_\alpha$ only finitely often. Define $\sigma: \mathfrak{M} \cong \mathfrak{M}$ by $\sigma(\eta_0, \eta_1) = (\tau_0, \tau_1)$, where $\tau_0 = \eta_0$ and where $\tau_1(n) = \eta_1(n) + s(\eta_0(n)) \bmod 2$, for each $n < \lg(\eta)$. 

Then $\sigma$ is clearly an automorphism of $\mathfrak{M}$, and moreover restricts to an isomorphism from $\mathfrak{A}$ to $\mathfrak{A}_{\mathcal{F}}$.
\end{proof}

\begin{lemma}
$\mathfrak{A}$ is the unique atomic model of $T$, and is furthermore prime.
\end{lemma}
\begin{proof}
Fix $\mathfrak{N} = (N, U^*_{\eta_0}, \pi^*_\alpha, \boldsymbol{\eta}^*: \eta \in J, \alpha \leq \omega) \models T$. We can suppose $acl(\emptyset)^{\mathfrak{N}} = X$. I find some $\mathcal{F}$ such that $\mathfrak{A}_{\mathcal{F}}$ embeds $\mathfrak{N}$, which suffices to show that $\mathfrak{A}$ is prime.

Indeed, for each $\alpha \in S$, choose $a_\alpha \in U^*_{\nu_\alpha}$. Let $f_\alpha \in 2^\omega$ be defined by $f_\alpha(n) = \pi^*_{m}(a_\alpha)(n)$ for some (any) $m > n$.

Let $\mathcal{F} = (f_\alpha: \alpha \in S)$. Then by Lemma~\ref{unifTypes}, the map $\sigma_0: (\nu_\alpha, f_\alpha) \mapsto a_\alpha$ is a partial elementary map from $\mathfrak{A}_{\mathcal{F}}$ to $\mathfrak{N}$. So it extends to a partial elementary map $\sigma: dcl((\nu_\alpha, f_\alpha): \alpha \in S)^{\mathfrak{A}_\mathcal{F}} \to dcl(a_\alpha: \alpha \in S)^{\mathfrak{N}}$. But then clearly $\sigma$ has domain all of $A_{\mathcal{F}}$. Hence $\sigma: \mathfrak{A}_{\mathcal{F}} \preceq \mathfrak{N}$.

To see that $\mathfrak{A}$ is the unique atomic model of $T$, note that if $\mathfrak{N}$ is atomic, then $\sigma$ is also surjective, again by Lemma~\ref{unifTypes}.
\end{proof}

We conclude the proof of Theorem~\ref{unifTheorem} with the following
 
\begin{lemma}
$\mathfrak{A}$ is not constructible.
\end{lemma}
\begin{proof}
Suppose $(\eta^\alpha: \alpha < \omega_1)$ were a construction of $\mathfrak{A}$ (it suffices to consider this order type by Theorem~\ref{constructibleModels}). Let $J_\alpha = \{\eta^\beta: \beta < \alpha\}$. Let $C \subset \omega_1$ be the club set of all $\alpha < \omega_1$ such that $J_\alpha = \{\eta \in J: \mbox{sup}(\eta_0) < \alpha\}$. Choose $\alpha \in S \cap C$. Let $\beta \geq \alpha$ be least with $\eta^\beta_0 = \nu_\alpha$, i.e. with $\eta^\beta \in U_{\nu_\alpha}$.

By Lemma~\ref{unifTypes}, it is clear that for any set $B$ such that $B \supset \{\eta \in J: \eta_0 \subset \nu_\alpha, \lg(\eta) < \omega\}$ and $B \cap U_{\nu_\alpha} = \emptyset$, that $tp(\eta^\beta / B)$ is nonisolated. In particular $tp(\eta^\beta / J_\beta)$ is nonisolated.
\end{proof}

\begin{remark}\label{ThirdExample}
Second Example: rewind back to the beginning of the section, and define instead $J_1$ to be the entire space $2^{\leq \omega}$. Then we  have without any special combinatorics that $\mathfrak{A}$ is the unique atomic model of $T$, and is prime, but is not constructible (although the language $\mathcal{L}$ now has size continuum).

\end{remark}

\section{Producing Many Atomic Models of Size $\aleph_1$}
\label{cons}

\subsection{Setup}
Fix throughout this section a complete theory $T$ in a language $\mathcal{L}$ of cardinality $\aleph_1$, such that $T$ has atomic models. Write $\mathcal{L} = \bigcup_{\alpha < \omega_1} \mathcal{L}_\alpha$ as the union of a continuous increasing chain of countable languages, and let $T_\alpha = T \restriction_{\mathcal{L}_\alpha}$. 

Recall that an $\mathcal{L}$-formula $\phi(\overline{x})$ is $T$-complete if it is consistent with $T$ and for every formula $\psi(\overline{x})$, $T + \phi(\overline{x})$ decides $\psi(\overline{x})$; equivalently, $\phi(\overline{x})$ isolates a single point in the Stone space $S^n(\emptyset)$. Since $T$ has atomic models, for every $\mathcal{L}$-formula $\phi(\overline{x})$ consistent with $T$, there is a $T$-complete formula $\psi(\overline{x})$ that implies $\phi(\overline{x})$.  So we can choose a club set $\mathbf{C}_0 \subseteq \omega_1$ such that for every $\alpha \in \mathbf{C}_0$ and for every $\mathcal{L}_\alpha$-formula $\phi(\overline{x})$, if $\phi(\overline{x})$ is consistent with $T$ then $\phi(\overline{x})$ has a $T$-complete extension $\psi(\overline{x})$, which is itself an $\mathcal{L}_\alpha$ formula.

It follows that for each $\alpha \in \mathbf{C}_0$, $T_\alpha$ has atomic models (though possibly not uncountable atomic models); and an $\mathcal{L}_\alpha$-formula $\phi(\overline{x})$ is $T$-complete if and only if it is $T_\alpha$ complete.

Let $\mathfrak{C}$ be a monster model of $T$. We use standard model-theoretic notation: $A, B, C,...$ will range over parameter sets, and $M, N, ...$ will range over elementary submodels of $\mathfrak{C}$. If $\overline{a} \in \mathfrak{C}$ and $\phi(\overline{x})$ is an $\mathcal{L}$-formula we write $\models \phi(\overline{a})$ for $\mathfrak{C} \models \phi(\overline{a})$. If $A \subset \mathfrak{C}$ is a set then $S^n(A)$ denotes the space of $n$-types over $A$, and $S(A)$ denotes $\bigcup_n S^n(A)$. If we write $f: A \to B$ it is implied that $f$ is partial $\mathcal{L}$-elementary.

Define an \emph{atomic set} to be a countable set $A \subset \mathfrak{C}$ such that every tuple $\overline{a} \in A$ has $tp(\overline{a})$ isolated by a single $\mathcal{L}$-formula. Say that $A$ is an $\alpha$-\emph{atomic set} if moreover this formula can be chosen in $\mathcal{L}_\alpha$.

If $A$ is $\alpha$-atomic, for some $\alpha \in \mathbf{C}_0$, then say that $A$ is an $\alpha$-\emph{base} if $A \restriction_{\mathcal{L}_\alpha} \preceq \mathfrak{C} \restriction_{\mathcal{L}_\alpha}$. $A$ is a \emph{base} if it is an $\alpha$-base for some $\alpha \in \mathbf{C}_0$. (Here we are using the term base as in ``amalgamation base.")

Note that for each $\alpha \in \mathbf{C}_0$, $\alpha$-bases exist and are unique up to isomorphism. Also, if $f: \omega \to \mathbf{C}_0$ is increasing, and $A_n$ is an increasing chain of $f(n)$-bases, then $\bigcup_n A_n$ is a $\bigcup_n f(n)$-base. Similarly, if $f: \omega_1 \to \mathbf{C}_0$ is increasing and cofinal, and $A_{\alpha}$ is an increasing chain of $f(\alpha)$-bases, then $\bigcup_n A_n$ is an atomic model of $T$.

For each atomic set $A$ and for each $n$, let $S^n_{at}(A)$ be the set of all atomic types over $A$ (i.e. all types $p(\overline{x}) \in S^n(A)$ such that whenever $\overline{a}$ realizes $p(\overline{x})$, $A \overline{a}$ is atomic). This is a dense subset of $S^n(A)$; give it the subspace topology. Let $S_{at}(A) = \bigcup_n S^n_{at}(A)$ with the disjoint union topology. 

For each atomic set $A$ and for each $\beta \in \mathbf{C}_0$, define $S^{n, \beta}_{at}(A)$ to be the set of all types $p(\overline{x}) \in S^n_{at}(A)$, such that whenever $\overline{a}$ realizes $p(\overline{x})$, $A \overline{a}$ is $\beta$-atomic. Give $S^{n, \beta}_{at}(A)$ the subspace topology.

\begin{lemma}
\begin{itemize}

\item For each $\beta \in \mathbf{C}_0$, the topology on $S^{n, \beta}_{at}(A)$ is generated by the $\mathcal{L}_\beta(A)$ formulas.

\begin{proof}
Let $\mathcal{O}$ be a basic open subset of $S^{n}_{at}(A)$; say $\mathcal{O} = \{p(\overline{x}) \in S^n_{at}(A): p(\overline{x}) \models \phi(\overline{x}, \overline{a})\}$ where $\phi(\overline{x}, \overline{a})$ is an $\mathcal{L}$-formula. Suppose $p(\overline{x}) \in S^{n, \beta}_{at}(A) \cap \mathcal{O}$. We can choose a complete $\mathcal{L}_\beta(\overline{a})$-formula $\psi(\overline{x}, \overline{a})$ such that $p(\overline{x}) \models \psi(\overline{x}, \overline{a})$. Let $\mathcal{U} = \{q(\overline{x}) \in S^n_{at}(A): q(\overline{x}) \models \psi(\overline{x}, \overline{a})\}$. Then $p \in \mathcal{U} \subseteq \mathcal{O}$ as desired.
\end{proof}

\item For each $\beta$, $S^{n, \beta}_{at}(A)$ is a Polish space (or empty), and is closed in $S^n_{at}(A)$.

\begin{proof}
Closure is clear. To see that it is a Polish space, let $\mathfrak{C}' = \mathfrak{C} \restriction \mathcal{L}_\beta$ and let $X = S^n(A)$ computed in $\mathfrak{C}'$. Then $X$ is a Polish space and $S^{n, \beta}_{at}(A)$ is naturally embedded as a $G_\delta$ subset of $X$.
\end{proof}
\end{itemize}
\end{lemma}

Let $\mathbf{K}_T$ be the class of atomic models of $T$. We now define what it means for $\mathbf{K}_T$ to be club totally transcendental: 

For $\alpha \in \mathbf{C}_0$, say that $\mathbf{K}_T$ is \emph{totally transcendental at} $\alpha$ if, letting $A$ be any $\alpha$-base, we have that $S_{at}(A)$ is scattered, i.e. has no perfect subset. Equivalently $\mathbf{K}_T$ is totally transcendental at $\alpha$ if for each $n$ and for each $\beta \in \mathbf{C}_0$, $S^{n, \beta}_{at}(A)$ is countable. Let the \emph{transcendence spectrum} of $\mathbf{K}_T$, $\mbox{Spec}_{\mathbf{K}_T}(t.t.)$, denote the set of all $\alpha \in \mathbf{C}_0$ at which $\mathbf{K}_T$ is totally transcendental.

\begin{definition}
$\mathbf{K}_T$ is \emph{club totally transcendental} if $\mbox{Spec}_{\mathbf{K}_T}(t.t.)$ contains a club.
\end{definition}

We aim to prove:

\vspace{2 mm}

\noindent \textbf{Theorem~\ref{main}.}
Suppose $\Phi^*$ holds and $\mathsf{Cov}(\mathcal{K}) \geq \aleph_2$. Suppose further that $\mathbf{K}_T$ is not club totally transcendental. Then $T$ has $2^{\aleph_1}$ atomic models of size $\aleph_1$.

\subsection{Club Totally Transcendental Property and the Existence of Constructible Models}\label{ttAndCons}

Given $\alpha \in \mathbf{C}_0$, note that if $\mathbf{K}_T$ is totally transcendental at $\alpha$, then the isolated types are dense in $S_{at}(A)$, where $A$ is any $\alpha$-base. The converse of course can fail drastically: say $T$ has $\aleph_1$-many sorts, each a model of DLO.

We can relate all this to constructible models as follows. Let the constructible spectrum of $\mathbf{K}_T$, $\mbox{Spec}_{\mathbf{K}_T}(CS)$, be the set of all $\alpha \in \mathbf{C}_0$ such that the isolated types are dense in $S_{at}(A)$, where $A$ is any $\alpha$-base. So by the preceding, $\mbox{Spec}_{\mathbf{K}_T}(t.t.) \subseteq \mbox{Spec}_{\mathbf{K}_T}(CS)$. Moreover:

\begin{theorem}\label{reduction}
$T$ has a constructible model if and only if $\mbox{Spec}_{\mathbf{K}_T}(CS)$ contains a club. In particular, if $\mathbf{K}_T$ is club totally transcendental, then $T$ has a constructible model.
\end{theorem}
\begin{proof}
First suppose $T$ has a constructible model $M$; say $M = (a_\alpha: \alpha < \omega_1)$ is a construction (with repetitions if $M$ is countable). Let $A_\alpha := \{a_\beta: \beta < \alpha\}$. Then the set $C = \{\alpha \in \mathbf{C}_0: A_\alpha \mbox{ is an $\alpha$-base}\}$ is club. Let $\alpha \in C$. Then $M$ is atomic over $A_\alpha$ by Theorem~\ref{constructibleModels} (applied to the theory $T(c_a: a \in A_\alpha)$ in the language $\mathcal{L}(c_a: a \in A_\alpha)$ where we add constants for elements of $A_\alpha$), which shows that the isolated types are dense in $S_{at}(A)$. Hence $\alpha \in \mbox{Spec}_{\mathbf{K}_T}(CS)$, so $\mbox{Spec}_{\mathbf{K}_T}(CS)$ contains a club.

Conversely, suppose $\mbox{Spec}_{\mathbf{K}_T}(CS) \supseteq C$, $C$ a club; we can suppose $C \subseteq \mathbf{C}_0$. We define an increasing, continuous chain of atomic sets $(A_\gamma: \gamma \in C')$ where $C' = \{\gamma_\alpha: \alpha < \omega_1\} \subseteq C$ is club and each $A_\alpha$ is an $\alpha$-base. We will further have that for all $\alpha < \omega_1$, $A_{\gamma_{\alpha+1}}$ is atomic over $A_{\gamma_\alpha}$. Finally, for each $\alpha < \omega_1$, we will have a construction $A_{\gamma_\alpha} = (a_\beta: \beta < \gamma_\alpha)$. As implied by the notation, for $\alpha < \alpha'$, the construction of $A_{\gamma_\alpha}$ is an initial segment of the construction of $A_{\gamma_{\alpha'}}$.

Note that this will suffice, since setting $M := \bigcup_\alpha A_{\gamma_\alpha}$, we have $M$ is a constructible model of $T$, as witnessed by $(a_\beta: \beta < \omega_1)$. 

Let $\gamma_0$ be the least infinite element of $C$ and let $A_{\gamma_0} = (a_\beta: \beta < \gamma_0)$ be any $\gamma_0$-base. Take unions at limit stages.

Suppose we have defined $\gamma_\alpha$ and $A_{\gamma_\alpha} = (a_\beta: \beta < \gamma_\alpha)$. Write $A = A_{\gamma_\alpha}$. Since the isolated types are dense in $S_{at}(A)$ we can choose an $M^\alpha \models T$, $M^\alpha \supseteq A$ and $M^\alpha$ atomic over $A$ by Theorem~\ref{atomicModelsAtAleph1} (applied to the theory $T(c_a: a \in A_\alpha)$ in the language $\mathcal{L}(c_a: a \in A_\alpha)$). It is possible that $M^\alpha$ is countable or even $M^\alpha = A$, but in any case we can enumerate $M^\alpha = (a_\beta^\alpha: \beta < \omega_1)$ so that for all $\beta < \gamma_\alpha$, $a^\alpha_\beta = a_\beta$. For each $\delta < \omega_1$, let $B_\delta = \{a^\alpha_\beta: \beta < \delta\}$.  Then the set of all $\delta$ such that $B_\delta$ is a $\delta$-base is club, so we can choose some such $\delta$ with $\delta \in C$ and $\delta > \gamma_\alpha$. Let $\gamma_{\alpha+1} = \delta$ and let $A_{\gamma_\alpha+1} = B_\delta$ and define $a_\beta = a^\alpha_\beta$ for all $\gamma_\alpha \leq \beta < \gamma_{\alpha+1}$.
\end{proof}
Hence, as a corollary of Theorem~\ref{main}, we will get

\vspace{2 mm}

\noindent \textbf{Theorem~\ref{submain}.} Suppose $\Phi^*$ holds and $\mathsf{Cov}(\mathcal{K}) \geq \aleph_2$. Suppose further that $T$ has no constructible models. Then $T$ has $2^{\aleph_1}$ atomic models of size $\aleph_1$.

\subsection{Club Totally Transcendental Property and Amalgamation}\label{amalgamation}

The proof of the main theorem will split into two cases: first, where $\mathbf{K}_T$ fails the club amalgamation property (to be defined below), and second, where $\mathbf{K}_T$ has the club amalgamation property but is not club totally transcendental. As in the countable language case we will actually have that if $\mathbf{K}_T$ is club totally transcendental then $\mathbf{K}_T$ has the club amalgamation property.

Let $\alpha \in \mathbf{C}_0$. An amalgamation problem at $\alpha$ is a triple $(A_0, A_1, A_2)$ of (countable) atomic sets, such that $A_0$ is an $\alpha$-base, and $A_0 \subseteq A_i$ for $i = 1, 2$. A solution is a triple $(A_3, f_1, f_2)$ such that $A_3$ is an atomic set, $f_i: A_i \to A_3$ are elementary, and the $f_i$'s agree on $A_0$. We say that $\mathbf{K}_T$ has the amalgamation property at $\alpha$ if every amalgation problem at $\alpha$ has a solution. We let the amalgamation spectrum of $\mathbf{K}_T$, $\mbox{Spec}_{\mathbf{K}_T}(AP)$, denote the set of all $\alpha \in \mathbf{C}_0$ at which $\mathbf{K}_T$ has the amalgamation property.

\begin{definition}
$\mathbf{K}_T$ has the club amalgamation property if $\mbox{Spec}_{\mathbf{K}_T}(AP)$ contains a club.
\end{definition}

\begin{theorem}
Suppose $\mathbf{K}_T$ is club totally transcendental. Then $\mathbf{K}_T$ has the club amalgamation property.
\end{theorem}
\begin{proof}
We show that $\mbox{Spec}_{\mathbf{K}_T}(t.t.) \subseteq \mbox{Spec}_{\mathbf{K}_T}(AP)$, which suffices.

Indeed, let $\alpha \in \mbox{Spec}_{\mathbf{K}_T}(t.t.)$ and let $(A_0, A_1, A_2)$ be an amalgamation property at $\alpha$. Choose $\beta > \alpha$ so that $\beta \in \mbox{Spec}_{\mathbf{K}_T}(t.t.)$ and each $A_i$ is $\beta$-atomic. Now $S^{\beta}_{at}(A_0)$ is countable, hence the isolated types are dense in $S^{\beta}_{at}(A_0)$. So by applying Theorem~\ref{countLangAmalg} to the theory $T \restriction_{\mathcal{L}_\beta}$ we get a solution.
\end{proof}
\subsection{Promises}

Our idea for constructing many models is the following: we will produce a tree $(A_{\mathbf{s}}: \mathbf{s} \in 2^{<\omega_1})$ of bases, such that if we set $M_\eta := \bigcup_{\alpha} A_{\eta \restriction_\alpha}$ for $\eta \in 2^{\omega_1}$, then each $M_\eta$ is an atomic model of $T$. 

We will also be producing, for each $\mathbf{s} \in 2^{<\omega_1}$, a set $\Phi_{\mathbf{s}} \subseteq S_{at}(A_{\mathbf{s}})$, such that every $\eta \supseteq \mathbf{s}$ has $M_\eta$ omits $\Phi_{\mathbf{s}}$. I.e. we are ``promising" to omit these types. Typically $\Phi_{\mathbf{s}}$ will be a union of $\aleph_1$-many closed nowhere dense sets, so in order to omit it we will need $\mathsf{Cov}(\mathcal{K}) \geq \aleph_2$.

By an appropriate failure of amalgamation, we will have that for each $\mathbf{s} \in 2^{<\omega_1}$, there is no $M, f_0, f_1$ such that: $f_i: A_{\mathbf{s}^\frown(i)} \to M$, $f_0 \restriction_{A_{\mathbf{s}}} = f_1 \restriction_{A_{\mathbf{s}}}$, and $M$ omits each $f_i(\Phi_{\mathbf{s}^\frown(i)})$. 

Then we will apply a diagonalization argument using $\Phi^*$ to get that $\{M_{\eta}: \eta \in 2^{\omega_1}\}$ contains $2^{\aleph_1}$ distinct isomorphism types.

In this subsection we develop some general machinery for building the tree $(A_{\mathbf{s}}, \Phi_{\mathbf{s}}: \mathbf{s} \in 2^{<\omega_1})$ and extracting $2^{\aleph_1}$ models of size $\aleph_1$. For the following, the reader should note that the special case $\mathbb{P} = \emptyset$ is actually an important example.
\begin{definition}
A \emph{system of promises} is a set $\mathbb{P}$ such that:

\begin{itemize}
\item Every $\Gamma \in \mathbb{P}$ is a nonempty subset of $S_{at}(A)$ for a (unique) base $A$. Write $A = \mbox{dom}(\Gamma)$.

If $A$ is an atomic set and $\Gamma \in P$, then say that $A$ omits $\Gamma$ if $A \supseteq \mbox{dom}(\Gamma)$ and for all $\overline{a} \in A$, $tp(\overline{a} / \mbox{dom}(\Gamma)) \not \in \Gamma$. If $\Phi \subseteq \mathbb{P}$ is countable and $A$ is an atomic set then say that $A$ omits $\Phi$ if $A$ omits $\Gamma$ for all $\Gamma \in \Phi$.

\item (Invariance) $\mathbb{P}$ is closed under $\mbox{Aut}(\mathfrak{C})$.

\item (Extendibility) Suppose $A$ is an atomic set, and $\Phi \subseteq \mathbb{P}$ is countable such that $A$ omits $\Phi$. Then for arbitrarily large $\alpha \in \mathbf{C}_0$ there is an $\alpha$-base $B \supseteq A$ such that $B$ omits $\Phi$.

\end{itemize}
\end{definition}

Suppose $\mathbb{P}$ is a system of promises. Then a $\mathbb{P}$-\emph{atomic set} ($\mathbb{P}$-\emph{base}, $(\alpha, \mathbb{P})$-\emph{atomic set}, $(\alpha, \mathbb{P})$-\emph{base}) is a pair $(A, \Phi)$ where $A$ is an atomic set (base, $\alpha$-atomic set, $\alpha$-base) and $\Phi \subset \mathbb{P}$ is countable and $A$ omits $\Phi$.

If $(A_0, \Phi_0)$ and $(A_1, \Phi_1)$ are $\mathbb{P}$-atomic sets, say that $(A_1, \Phi_1)$ extends $(A_0, \Phi_0)$, and write that $(A_0, \Phi_0) \subseteq (A_1, \Phi_1)$, if $A_0 \subseteq A_1$ and $\Phi_0 \subseteq \Phi_1$.

A $\mathbb{P}$-\emph{amalgamation problem} is a triple  of $\mathbb{P}$-atomic sets $(A_0, \Phi_0), (A_1, \Phi_1), (A_2, \Phi_2))$ where $(A_0, \Phi_0)$ is a $\mathbb{P}$-base and each $(A_i, \Phi_i)$ extends $(A_0, \Phi_0)$. We call $(A_0, \Phi_0)$ is called the base of the problem.

A \emph{solution} to the above problem is a sequence $((B, \Psi), f_1, f_2)$ where $(B, \Psi)$ is a $\mathbb{P}$-atomic set, and  $f_1: A_1 \to B$ and $f_2: A_2 \to B$ are both the identity on $A$, and $f_1(\Phi_1) \cup f_2(\Phi_2) \subseteq \Psi$.

Note that by the invariance property of promise systems, if two amalgamation problems are isomorphic then one has a solution if and only if the other does.

For $\alpha \in \mathbf{C}_0$, we say that $\mathbb{P}$ has the \emph{amalgamation property at} $\alpha$ if there is some $(\alpha, \mathbb{P})$-base $(A, \Phi)$, such that every $\mathbb{P}$-amalgamation problem with base $(A, \Phi)$ has a solution.

Let the \emph{amalgamation spectrum of} $\mathbb{P}$, $\mbox{Spec}_{\mathbf{K}_T}(\mathbb{P})$, be the set of $\alpha \in \mathbf{C}_0$ such that $\mathbb{P}$ has the amalgamation property at $\alpha$. We say that $\mathbb{P}$ \emph{has the club amalgamation property} if $\mbox{Spec}_{\mathbf{K}_T}(\mathbb{P})$ contains a club.

In particular, if $\mathbb{P} = \emptyset$ then $\mbox{Spec}_{\mathbf{K}_T}(\mathbb{P}) = \mbox{Spec}_{\mathbf{K}_T}(AP)$, and so $\mathbb{P}$ has the club amalgamation property iff $\mathbf{K}_T$ has the club amalgamation property.

The proof of the following (in a different context) is due originally to Shelah \cite{ShelahClass1}, see \cite{cat} Theorem 17.11 for a nice exposition.

\begin{lemma}\label{promises}
Suppose $\Phi^*$ holds, and $T$ admits a system of promises $\mathbb{P}$ which fails the club amalgamation property. Then $T$ has $2^{\aleph_1}$ atomic models of size $\aleph_1$. (In fact we just need $\Phi(\omega_1 \backslash \mbox{Spec}_{\mathbf{K}_T}(\mathbb{P}))$ to hold.)
\end{lemma}

The rest of this subsection is a proof of the lemma. Note that $T$ has at most one countable atomic model, so it suffices to show that $T$ has $2^{\aleph_1}$ atomic models of size $\leq \aleph_1$.

Let $S = \mathbf{C}_0 \backslash \mbox{Spec}_{\mathbf{K}_T}(\mathbb{P})$. We are assuming that $S$ is stationary; thus $\Phi(S)$ holds (and in particular $2^{\aleph_0} < 2^{\aleph_1}$).

The proof splits into two cases.

\vspace{2 mm}

\noindent \textbf{Case A.} There exist $\mathbb{P}$-bases $(A_0, \Phi_0) \subseteq (A, \Phi)$, such that for every $\mathbb{P}$-base $(B, \Psi) \supseteq (A, \Phi)$, there exist $\mathbb{P}$-bases $(B_0, \Psi_0)$ and $(B_1, \Psi_1)$ extending $(B, \Psi)$, such that the $\mathbb{P}$-amalgamation problem $((A_0, \Phi_0), (B_0, \Psi_0), (B_1, \Psi_1))$ has no solution.

In this case we build inductively a system $(A_{\mathbf{s}}, \Phi_{\mathbf{s}}, \alpha_{\mathbf{s}}: \mathbf{s} \in 2^{<\omega_1})$ such that:

\begin{itemize}
\item For each $\mathbf{s} \in 2^{<\omega_1}$, $\alpha_{\mathbf{s}} \in \mathbf{C}_0$ and $(A_{\mathbf{s}}, \Phi_{\mathbf{s}})$ is an $(\alpha_{\mathbf{s}}, \mathbb{P})$-base.

\item For $\mathbf{s} \subseteq \mathbf{t}$, $\alpha_{\mathbf{s}} < \alpha_{\mathbf{t}}$, and $(A_{\mathbf{s}}, \Phi_{\mathbf{s}}) \subseteq (A_{\mathbf{t}}, \Phi_{\mathbf{t}})$.

\item For each $\mathbf{s} \in 2^{<\omega_1}$ of limit length, $\alpha_{\mathbf{s}} = \bigcup_{\mathbf{t} \subset \mathbf{s}} \alpha_{\mathbf{t}}$ and $A_{\mathbf{s}} = \bigcup_{\mathbf{t} \subset \mathbf{s}} A_{\mathbf{t}}$ and $\Phi_{\mathbf{s}} = \bigcup_{\mathbf{t} \subset \mathbf{s}} \Phi_{\mathbf{t}}$.

\item For each $\mathbf{s} \in 2^{<\omega_1}$, the $\mathbb{P}$-amalgamation problem $((A_0, \Phi_0), (A_{\mathbf{s}^\frown(0)}, \Phi_{\mathbf{s}^\frown(0)}), (A_{\mathbf{s}^\frown(1)}, \Phi_{\mathbf{s}^\frown(1)}))$ has no solution.
\end{itemize}

For each $\eta \in 2^{\omega_1}$, let $M_\eta = \bigcup_{\alpha < \omega_1} A_{\eta \restriction_\alpha}$, an atomic model of $T$ of size $\leq \aleph_1$. Then for each $\eta \not= \tau$, $(M_\eta, a: a \in A_0) \not \cong (M_\tau, a: a \in A_0)$. Hence $\{M_\eta: \eta \in 2^{\aleph_1}\}$ represents $2^{\aleph_1}$ different isomorphism types if we add countably many constants. Since $2^{\aleph_0} < 2^{\aleph_1}$ it follows that $\{M_\eta: \eta \in 2^{\aleph_1}\}$ represents $2^{\aleph_1}$ different isomorphism types.

\vspace{2 mm}

\noindent \textbf{Case B.} (The negation of Case A.) For all $\mathbb{P}$-bases $(A_0, \Phi_0) \subseteq (A, \Phi)$ there is $(B, \Psi) \supseteq (A, \Phi)$ such that for all $(B_0, \Psi_0)$ and $(B_1, \Psi_1)$ extending $(B, \Psi)$, the $\mathbb{P}$-amalgamation problem $((A_0, \Phi_0), (B_0, \Psi_0), (B_1, \Psi_1))$ has a solution.

In this case we inductively build a system $(A_{\mathbf{s}}, \Phi_{\mathbf{s}}, \alpha_{\mathbf{s}} : \mathbf{s} \in 2^{<\omega_1})$ such that:

\begin{itemize}
\item For each $\mathbf{s} \in 2^{<\omega_1}$, $\alpha_{\mathbf{s}} \in \mathbf{C}_0$ and $(A_{\mathbf{s}}, \Phi_{\mathbf{s}})$ is an $(\alpha_{\mathbf{s}}, \mathbb{P})$-base.

\item For $\mathbf{s} \subseteq \mathbf{t}$, $\alpha_{\mathbf{s}} < \alpha_{\mathbf{t}}$, and $(A_{\mathbf{s}}, \Phi_{\mathbf{s}}) \subset (A_{\mathbf{t}}, \Phi_{\mathbf{t}})$.

\item For each $\mathbf{s} \in 2^{<\omega_1}$ of limit length, $\alpha_{\mathbf{s}} = \bigcup_{\mathbf{t} \subset \mathbf{s}} \alpha_{\mathbf{t}}$ and $A_{\mathbf{s}} = \bigcup_{\mathbf{t} \subset \mathbf{s}} A_{\mathbf{t}}$ and $\Phi_{\mathbf{s}} = \bigcup_{\mathbf{t} \subset \mathbf{s}} \Phi_{\mathbf{t}}$.

\item For each $\mathbf{s} \in 2^{<\omega_1}$ with $\alpha_{\mathbf{s}} \in S$, the $\mathbb{P}$-amalgamation problem  $((A_{\mathbf{s}}, \Phi_{\mathbf{s}}), (A_{\mathbf{s}^\frown(0)}, \Phi_{\mathbf{s}^\frown(0)})$, $(A_{\mathbf{s}^\frown(1)}, \Phi_{\mathbf{s}^\frown(1)}))$ has no solution.

\item For each $\mathbf{s} \in 2^{<\omega_1}$, for each $i \in 2$ and for each pair of $\mathbb{P}$-bases $(B_0, \Psi_0)$ and $(B_1, \Psi_1)$ extending $(A_{\mathbf{s}^\frown(i)}, \Phi_{\mathbf{s}^\frown(i)})$, the $\mathbb{P}$-amalgamation problem $((A_{\mathbf{s}}, \Phi_{\mathbf{s}}), (B_0, \Psi_0), (B_1, \Psi_1))$ has a solution. 
\end{itemize}

For each $\eta \in 2^{\omega_1}$ let $M_\eta := \bigcup_{\alpha < \omega_1} A_{\mathbf{s} \restriction_\alpha}$, an atomic model of $T$ of size $\leq \aleph_1$. I claim that in fact each $M_\eta$ has size exactly $\aleph_1$. Indeed, fix $\eta \in 2^{\omega_1}$. Then there are uncountably many $\alpha < \omega_1$ such that $\alpha = \alpha_{\eta \restriction_{\alpha}} \in S$, so it suffices to show that for each such $\alpha$, $A_{\eta \restriction_\alpha}$ is strictly contained in $A_{\eta \restriction_{\alpha+1}}$. Suppose not; set $A = A_{\eta \restriction_\alpha} = A_{\eta \restriction_{\alpha+1}}$ and set $\Phi = \Phi_{\eta \restriction_{\alpha+1}} \supseteq \Phi_{\eta \restriction_\alpha}$. Then $(A, \Phi)$ is an $(\alpha, \mathbb{P})$ base, but every $\mathbb{P}$-amalgamation problem with base $(A, \Phi)$ must have a solution by the final requirement above, contradicting $\alpha \in S$.

Choose bijections $\sigma_\eta: M_\eta \to \omega_1$, such that for all $\eta, \tau \in 2^{\omega_1}$ with $\eta \restriction_\alpha = \tau \restriction_\alpha = \mathbf{s}$ say, we have that $\sigma_\eta \restriction_{A_{\mathbf{s}}} = \sigma_\tau \restriction_{A_{\mathbf{s}}} := \sigma_{\mathbf{s}}$. 

We view $(2 \times 2 \times \omega_1)^{<\omega_1}$ as a subset of $2^{<\omega_1} \times 2^{<\omega_1} \times \omega_1^{<\omega_1}$. Define $F: (2 \times 2 \times \omega_1)^{<\omega_1} \to 2$ by  $F(\mathbf{s}, \mathbf{t}, h) = 1$ if:

\begin{itemize}
\item $\sigma_{\mathbf{s}}(A_{\mathbf{s}}) = \sigma_{\mathbf{t}}(A_{\mathbf{t}}) = \alpha_{\mathbf{s}} = \alpha_{\mathbf{t}} = \lg(\mathbf{s}) = \lg(\mathbf{t}) =: \alpha$ say.
\item $h: \alpha \to \alpha$ is a bijection.
\item $\sigma_{\mathbf{t}}^{-1} \circ h \circ \sigma_{\mathbf{s}}: A_{\mathbf{s}} \cong A_{\mathbf{t}}$.
\item For some or any extension $g$ of $\sigma_{\mathbf{t}}^{-1} \circ h \circ \sigma_{\mathbf{s}}$ to $A_{\mathbf{s}^\frown(0)}$, the $\mathbb{P}$-amalgamation problem $((A_{\mathbf{t}}, g(\Phi_{\mathbf{s}}) \cup \Phi_{\mathbf{t}}), (g(A_{\mathbf{s}^\frown(0)}), g(\Phi_{\mathbf{s}^\frown(0)}) \cup \Phi_{\mathbf{t}}), (A_{\mathbf{t}^\frown(0)}, g(\Phi_{\mathbf{s}}) \cup \Phi_{\mathbf{t}^\frown(0)}))$ has a solution.
\end{itemize}
$F(\mathbf{s}, \mathbf{t}, h) = 0$ else.

Now choose disjoint stationary subsets $(S_\alpha: \alpha < \omega_1)$ of $S$ ; then $\Phi(S_\alpha)$ holds for each $\alpha < \omega_1$. For each $\alpha < \omega_1$ choose $g_\alpha: S_\alpha \to 2$ such that for every $(\eta, \tau, f) \in (2 \times 2 \times \omega_1)^{\omega_1}$, the set of all $\beta \in S_\alpha$ with $g_\alpha(\beta) = F(\eta \restriction_\beta, \tau \restriction_\beta, f \restriction_\beta)$ is stationary.

For $X \subset \omega_1$ define $\eta_X : \omega_1 \to 2$ by: $\eta_X(\beta) = g_\alpha(\beta)$ if $\beta \in S_\alpha$ and $\alpha \in X$, and $\eta_X(\beta) = 0$ else.

I claim that for all $X \not= Y$, $M_X \not \cong M_Y$, which suffices.

Indeed, suppose $X \not= Y$ and yet $f: \omega_1 \to \omega_1$ is a bijection with $\phi := \sigma_{\tau}^{-1} \circ f \circ \sigma_{\eta}: M_\eta \cong M_\tau$. We can suppose $\alpha \in X \backslash Y$. Let $\eta = \eta_X$ and let $\tau = \tau_Y$. Let $C$ be the club set of all $\beta < \omega_1$ such that $\sigma_{\eta \restriction_\beta}(A_{\eta \restriction_\beta}) = \sigma_{\tau \restriction_\beta}(A_{\tau \restriction_\beta}) = \alpha_{\eta \restriction_\beta} = \alpha_{\tau \restriction_\beta} = f[\beta]= \beta$. 

Choose $\beta \in S_\alpha \cap C$ such that $F(\eta \restriction_\beta, \tau \restriction_\beta, f \restriction_\beta) = g_\alpha(\beta)$. Write $\mathbf{s} = \eta \restriction_\beta$, $\mathbf{t} = \tau \restriction_\beta$, $h = f \restriction_\beta$.

Note that the first two items of the definition of $F$ are met, so $F(\mathbf{s}, \mathbf{t}, h) = 1$ iff the third item holds. Also note that $\tau(\beta) = 0$.
There are two cases:

\vspace{2 mm}

\noindent \textbf{Case B0.} $F(\mathbf{s}, \mathbf{t}, h) = 0$. Then $\eta(\beta) = 0$. But then clearly the isomorphism $\phi: M_\eta \cong M_\tau$ witnessses that the $\mathbb{P}$-amalgamation problem in the third item of the definition of $F$ has a solution, contradicting the case.

\vspace{2 mm}

\noindent \textbf{Case B1.} $F(\mathbf{s}, \mathbf{t}, h) = 1$. Then $\eta(\beta) = 1$. Let $g$ be any extension of $\phi$ to $A_{\mathbf{s}^\frown(0)}$; then we can choose a solution $(i_0, i_1, (B, \Psi_0))$ to the $\mathbb{P}$-amalgamation problem $((A_{\mathbf{t}}, g(\Phi_{\mathbf{s}}) \cup \Phi_{\mathbf{t}}), (g(A_{\mathbf{s}^\frown(0)}), g(\Phi_{\mathbf{s}^\frown(0)}) \cup \Phi_{\mathbf{t}}), (A_{\mathbf{t}^\frown(0)}, g(\Phi_{\mathbf{s}}) \cup \Phi_{\mathbf{t}^\frown(0)}))$, where moreover $i_1: A_{\mathbf{t}^\frown(0)} \to B$ is the inclusion. We can use the isomorphism $\phi$ to get a solution $(j_0, j_1, (C, \Psi_1))$ to the $\mathbb{P}$-amalgamation problem $((A_{\mathbf{t}}, \phi(\Phi_{\mathbf{s}}) \cup \Phi_{\mathbf{t}}), (\phi(A_{\mathbf{s}^\frown(1)}), \phi(\Phi_{\mathbf{s}^\frown(1)}) \cup \Phi_{\mathbf{t}}), (A_{\mathbf{t}^\frown(0)}, \phi(\Phi_{\mathbf{s}}) \cup \Phi_{\mathbf{t}^\frown(0)}))$ where again $j_1: A_{\mathbf{t}^\frown(0)} \to C$ is the inclusion. Then by the construction of the system $(A_\mathbf{s}, \Phi_{\mathbf{s}}, \alpha_{\mathbf{s}})$, the $\mathbb{P}$-amalgamation problem $((A_\mathbf{t}, \Phi_{\mathbf{t}}), (B, \Psi_0), (C, \Psi_1))$ has a solution. But this yields a solution to the $\mathbb{P}$-amalgamation problem $((A_{\mathbf{s}}, \Phi_{\mathbf{s}}), (A_{\mathbf{s}^\frown(0)}, \Phi_{\mathbf{s}^\frown(0)})$, $(A_{\mathbf{s}^\frown(1)}, \Phi_{\mathbf{s}^\frown(1)}))$, contradiction.

\subsection{Proof of Theorem~\ref{main}}

Throughout this section, we suppose $\Phi^*$ holds and $\mathsf{Cov}(\mathcal{K}) \geq \aleph_2$, and $\mathbf{K}_T$ is not club totally transcendental. We aim to construct $2^{\aleph_1}$ atomic models of $T$ of size $\aleph_1$.

Recall that if we let $\mathbb{P} = \emptyset$ be the empty system of promises, then $\mbox{Spec}_{\mathbf{K}_T}(\mathbb{P}) = \mbox{Spec}_{\mathbf{K}_T}(AP)$; so if $\mathbf{K}_T$ fails the club amalgamation property then by Lemma~\ref{promises} we are done. Hence we can suppose that $\mathbf{K}_T$ has the club amalgamation property.

\begin{lemma}\label{typeAmalg}
If $\alpha \in \mbox{Spec}_{\mathbf{K}_T}(AP)$, $A$ is an $\alpha$-base, $B \supseteq A$ is atomic, and $p(\overline{x}) \in S_{at}(A)$, then $p(\overline{x})$ extends to a type in $S_{at}(B)$.
\end{lemma}
\begin{proof}
Let $\overline{a}$ be a realization of $p(\overline{x})$. Then the amalgamation problem $(A, A \overline{a}, B)$ has a solution, which is equivalent to the claim.
\end{proof}

\begin{lemma}
There is a club $\mathbf{C}_1 \subseteq \mbox{Spec}_{\mathbf{K}_T}(AP)$ and a number $n_0$ such that for every $\alpha < \beta$ both in $\mathbf{C}_1$, if $A$ is the $\alpha$-base then $S^{n_0, \beta}_{at}(A)$ has size continuum. In particular, $\mbox{Spec}_{\mathbf{K}_T}(t.t.) \cap \mathbf{C}_1 = \emptyset$.
\end{lemma}

\begin{proof}
We can choose some club $C \subseteq \mbox{Spec}_{\mathbf{K}_T}(AP)$ by assumption.

Let $\alpha \in C \backslash \mbox{Spec}_{\mathbf{K}_T}(t.t.)$. Let $A$ be an $\alpha$-base. Let $n_0$ be such that $S^{n_0}_{at}(A)$ has size continuum. Let $\beta > \alpha$ with $\beta \in C$ and let $B \supseteq A$ be a $\beta$-base. Then by Lemma~\ref{typeAmalg}, $S^{n_0}_{at}(B)$ has size continuum. It follows that $S^{n_0}_{at}(B)$ has a perfect subset, since otherwise we would have $|S^{n_0}_{at}(B)| \leq \aleph_1 < 2^{\aleph_0}$. Hence there is some $f(\beta)$ such that $S^{n_0, \beta}_{at}(B)$ has size continuum; we can choose $f(\beta) \in C \backslash \beta$. 

Let $\mathbf{C}_1$ be the club $(\alpha, f(\alpha), f^2(\alpha), \ldots, f^\gamma(\alpha), \ldots)$ (take unions at limit stages).
\end{proof}

For the rest of the proof, fix $\mathbf{C}_1, n_0$ as above. For each ordinal $\alpha < \omega_1$ let $\alpha^+$ denote the least ordinal $\beta > \alpha$ with $\beta \in \mathbf{C}_1$. So for any $\alpha \in \mathbf{C}_1$ and for any $\alpha$-base $A$, $S^{n_0, \alpha^+}_{at}(A)$ has size continuum. Let $K(A)$ denote the perfect kernel of $S^{n_0, \alpha^+}_{at}(A)$. 

The following definition gives a nice description of $K(A)$.

\begin{definition}
Let $\alpha \in \mathbf{C}_1$, and let $\phi(\overline{x}; \overline{y})$ be a partitioned $\mathcal{L}_{\alpha^+}$ formula with $|\overline{x}| = n_0$. Then say that $\phi(\overline{x}; \overline{y})$ is $\alpha$-\emph{unbounded} if for some (any) $\alpha$-base $A$, there is some $p(\overline{x}) \in K(A)$ and some $\overline{a} \in A$ with $\phi(\overline{x}; \overline{a}) \in p(\overline{x})$. Note that, since $K(A)$ is fixed under $A$-automorphisms, we have that $p(\overline{x}) \in K(A)$ if and only if for all $\phi(\overline{x}; \overline{a}) \in p(\overline{x})$, $\phi(\overline{x}; \overline{y})$ is $\alpha$-unbounded.
\end{definition}

Now fix for the time being $\alpha \in \mathbf{C}_1$ and an $\alpha$-base $A$. We identify closed subsets of $S^n_{at}(A)$ with the corresponding partial $n$-types over $A$. So for instance if $C \subseteq S^n_{at}(A)$ is closed then we write $C(\overline{x}) \models \phi(\overline{x})$ to indicate that for every $p(\overline{x}) \in C$, $\phi(\overline{x}) \in p(\overline{x})$. Let $\Phi_{at}(A)$ denote the subsets of $S_{at}(A)$ which are in fact closed subsets of $S^{n, \beta}_{at}(A)$ for some $n \in \omega$ and some $\beta \in \mathbf{C}_1$. For example, each $S^{n, \beta}_{at}(A) \in \Phi_{at}(A)$.

We define a (pre)-partial ordering $\leq$ on $\Phi_{at}(A)$, with the idea that $C \leq D$ means that if we realize $D$ over $A$, then it is hard to realize $C$ over $A$.

First we define the immediate successors of $\leq$:  

\begin{definition}\label{precDef} Let $C, F \in \Phi_{at}(A)$ be given. Then $C \prec F$ if and only if one of the following holds:

\begin{enumerate}
\item For some $\beta \geq \alpha$ in $\mathbf{C}_1$, $C(y, \overline{z})$ is a closed subset of $S^{1+n, \beta}_{at}(A)$, $F(\overline{z}, \overline{w})$ is a closed subset of $S^{n+m, \beta}_{at}(A)$, and there is some $\mathcal{L}_\beta$-formula $\phi(y, \overline{z}, \overline{w})$ such that $F(\overline{z}, \overline{w})$ is defined by the intersection of the following closed sets:

\begin{itemize}
\item $S^{n+m, \beta}_{at}(A)$;
\item $``\exists y \phi(y, \overline{z}, \overline{w})$;"
\item $``\forall y (\phi(y, \overline{z}, \overline{w}) \rightarrow \psi(y, \overline{z}, \overline{d}))$" for each formula $\psi(y, \overline{z}, \overline{d})$ with $C(y, \overline{z}) \models \psi(y, \overline{z}, \overline{d})$.
\end{itemize}

So, whenever $B \supseteq A$ is $\beta$-atomic and whenever $\overline{a} \overline{b} \in B^{n+m}$, then $tp(\overline{a} \overline{b}/A) \in F$ iff the following holds: there is some $q(y) \in S^{1, \beta}_{at}(B)$ with $q(y) \models \phi(y, \overline{a}, \overline{b})$, and moreover, for any such $q(y)$, if we let $r(y, \overline{z})$ be the set of all $\mathcal{L}(A)$-formulas $\psi(y, \overline{z})$ such that $q(y) \models \psi(y, \overline{a})$, then $r(y, \overline{z}) \in C$.

\item For some $\beta \geq \alpha$ in $\mathbf{C}_1$, $C(\overline{y}, \overline{z}) \subseteq S^{n_0 + n, \beta^+}_{at}(A)$ is closed, $F(\overline{z}, \overline{w})$ is a closed subset of $S^{n+m, \beta}_{at}(A)$, and there is some $\beta$-unbounded, complete formula $\phi(\overline{y}; \overline{z}, \overline{w})$ such that $F(\overline{z}, \overline{w})$ is defined by the intersection of the following closed sets:

\begin{itemize}
\item $S^{n+m, \beta}_{at}(A)$;
\item ``$\exists \overline{y} \phi(\overline{y}, \overline{z}, \overline{w});$"
\item ``$\forall \overline{y}  (\tau(\overline{y}, \overline{z}, \overline{w}, \overline{d}) \land \phi(\overline{y}, \overline{z}, \overline{w}) \rightarrow \psi(\overline{y}, \overline{z}, \overline{d})),"$ for all $\beta$-unbounded, complete formulas $\tau(\overline{y}; \overline{z}, \overline{w}, \overline{u})$, and all $\mathcal{L}(A)$-formulas $\psi(\overline{y}, \overline{z}, \overline{d})$ such that $|\overline{d}| = |\overline{u}|$ and such that $C(\overline{y}, \overline{z}) \models \psi(\overline{y}, \overline{z}, \overline{d})$.
\end{itemize}

So, whenever $B \supseteq A$ is a $\beta$-base and whenever $\overline{a} \overline{b} \in B^{n+m}$, we have $tp(\overline{a} \overline{b}/A) \in F$ iff the following holds: there is some $q(\overline{y}) \in K(B)$ with $q(\overline{y}) \models \phi(\overline{y}, \overline{a}, \overline{b})$, and moreover, for any such $q(\overline{y})$, if we let $r(\overline{y}, \overline{z})$ be the set of all $\mathcal{L}(A)$-formulas $\psi(\overline{y}, \overline{z})$ such that $q(\overline{y}) \models \psi(\overline{y}, \overline{a})$, then $r(\overline{y}, \overline{z}) \in C$.
\end{enumerate}
\end{definition}

Now let $\leq$ be the the least partial order containing $\prec$, i.e. $C \leq F$ iff there is a sequence $C = C_0 \prec C_1 \prec \ldots \prec C_{n-1} = F$.

Note that for each $C$, there are at most $\aleph_1$-many $F$ with $C \prec F$, and so there are at most $\aleph_1$-many $F$ with $C \leq F$.

Given $p(\overline{x}) \in S_{at}(A)$, let $\Gamma(p(\overline{x})) = \bigcup_{C \geq \{p(\overline{x})\}} C$, so this is the union of $\aleph_1$-many closed subsets of $S_{at}(A)$. 

Let $K^*(A) := \{p(\overline{x}) \in K(A): A \mbox{ omits } \Gamma(p(\overline{x}))\}$. Finally let $\mathbb{P} = \{\Gamma(p(\overline{x})): p(\overline{x}) \in K^*(A) \mbox{ where $A$ is an $\alpha$-base for some $\alpha \in \mathbf{C}_1$}\}$.

Then it suffices to establish that $\mathbb{P}$ is a system of promises, with $\mbox{Spec}_{\mathbf{K}_T}(\mathbb{P}) \cap \mathbf{C}_1 = \emptyset$. Towards this we prove the following three lemmas.

\begin{lemma} \label{amalg1} For every $\alpha \in \mathbf{C}_1$ and every $\alpha$-base $A$, $|K(A) \backslash K^*(A)| \leq \aleph_1$, in particular $K^*(A)$ is $\aleph_1$-comeager in $K(A)$.
\end{lemma}
\begin{proof}
Suppose $p(\overline{x}) \not \in K^*(A)$. Then there are closed sets  $\{p(\overline{x})\} = C_0 \prec C_1 \prec \ldots \prec C_n$, and some $\overline{a} \in A$ realizing $C_n$. Then from examining the definition of $\prec$, we see that we can recover $p(\overline{x})$ from $tp(\overline{a}/A)$ and from the formulas and ordinals witnessing that $C_i \prec C_{i+1}$ for $i < n$. There are only $\aleph_1$-many possibilities for the latter, and so there are only $\aleph_1$-many $p(\overline{x})$ not in $K^*(A)$.
\end{proof}

\begin{lemma}\label{extend}
Suppose $\alpha \leq \beta$ are both in $\mathbf{C}_1$, $A$ is an $\alpha$-base, $p(\overline{x}) \in K^*(A)$, and $B \supseteq A$ is a $\beta$-atomic set which omits $\Gamma(p(\overline{x}))$. Let $$X = \{q(y) \in S^{1, \beta}_{at}(B): B a \mbox{ omits } \Gamma(p(\overline{x})) \mbox{ for some (any) realization $a$ of $q(y)$}\}.$$ Then $X$ is $\aleph_1$-comeager in $S^{1, \beta}_{at}(B)$.
\end{lemma}
\begin{proof}
For each $q(y) \in S^{1, \beta}_{at}(B)$ and each $\overline{a} \in B$, let $[q, \overline{a}](y, \overline{z}) \in S^{1+|a|, \beta}_{at}(A)$ be the set of all $\mathcal{L}_\beta(A)$-formulas $\phi(y, \overline{z})$ such that $\phi(y, \overline{a}) \in q(y)$. 

Fix $\overline{a} \in B$, say $|\overline{a}| = n$, and fix $C \geq \{p(\overline{x})\}$ with $C$ a closed subset of $S^{1+n, \beta}_{at}(A)$. It suffices to show that $D := \{q(y) \in S^{1, \beta}_{at}(B): [q, \overline{a}](y, \overline{z}) \in C\}$ is closed nowhere dense. It is clearly closed, since $C$ is.

Suppose it weren't nowhere dense, say $\mathcal{O} = \{q(y) \in S^{1, \beta}_{at}(B): \phi(y; \overline{a}, \overline{b}) \in q(y)\}$ is such that $\emptyset \not= \mathcal{O} \subseteq D$.

Let $m = |\overline{b}|$ and let $\overline{w}$ be a tuple of variables of length $m$. Let $F(\overline{z}, \overline{w}) \subseteq S^{n+m, \beta}_{at}(A)$ be the closed set defined as in the first clause of Definition~\ref{precDef}.

Then $F \succ C$ so $F \subset \Gamma(p(\overline{x}))$, but $\overline{a} \overline{b}$ realizes $F$, contradiction.

\end{proof}

\begin{lemma}
\label{amalg2}
Suppose $\alpha \leq \beta$ are both in $\mathbf{C}_1$, $A$ is an $\alpha$-base, $p(\overline{x}) \in K^*(A)$, $B \supseteq A$ is a $\beta$-base, and $B$ omits $\Gamma(p(\overline{x}))$. Let 
$$X = \{q(\overline{y}) \in K(B): B \overline{d} \mbox{ omits $\Gamma(p(\overline{x}))$ for some (any) $\overline{d}$ realizing $q(\overline{y})$}\}.$$ Then $X$ is $\aleph_1$-comeager in $K(B)$.
\end{lemma}

\begin{proof}
For each $q(\overline{y}) \in K(B)$ and each $\overline{a} \in B$, let $[q, \overline{a}](\overline{y}, \overline{z}) \in S^{n_0 + |\overline{a}|, \beta^+}_{at}(A)$ be the set of all $\mathcal{L}_{\beta^+}(A)$-formulas $\phi(\overline{y}, \overline{z})$ such that $\phi(\overline{y}, \overline{a}) \in q(y)$. 

Fix $\overline{a} \in B$, say $|\overline{a}| = n$, and fix $C \subseteq S^{n_0 + n, \beta^+}_{at}(A)$ closed, with $C \geq \{p(\overline{x})\}$. It suffices to show that $D := \{q(\overline{y}) \in K(B): [q, \overline{a}](\overline{y}, \overline{z}) \in C\}$ is closed nowhere dense in $K(B)$. It is clearly closed, since $C$ is.

Suppose it weren't nowhere dense, say $\mathcal{O} = \{q(\overline{y}) \in K(B): \phi(\overline{y}; \overline{a}, \overline{b}) \in q(\overline{y})\}$ is such that $\emptyset \not= \mathcal{O} \subseteq D$. We can suppose $\phi(\overline{y}, \overline{z}, \overline{w})$ is complete.

Let $m = |\overline{b}|$ and let $\overline{w}$ be a tuple of variables of length $m$. Let $F(\overline{z}, \overline{w}) \subseteq S^{n + m, \beta}_{at}(A)$ be the closed set defined as in the second clause of Definition~\ref{precDef}.

Then $F \succ C$ so $F \subset \Gamma(p(\overline{x}))$, but $\overline{a} \overline{b}$ realizes $F$, contradiction.
\end{proof}

We conclude the proof of Theorem~\ref{main} with:

\begin{lemma}
$\mathbb{P}$ is a system of promises, and $\mbox{Spec}_{\mathbf{K}_T}(\mathbb{P}) \cap \mathbf{C}_1 = \emptyset$.
\end{lemma}
\begin{proof} Invariance for $\mathbb{P}$ is clear.

Extendibility follows from an iterated application of Lemma~\ref{extend}.

Finally, suppose $\alpha \in \mathbf{C}_1$ and $(A, \Phi)$ is an $(\alpha, \mathbb{P})$-base. Write $\Phi = \{\Gamma(p_n(\overline{x}_n)): n < \omega\}$, where $p_n(\overline{x}_n) \in K^*(A_n)$ for some $A_n \subseteq A$. Let $X = \{q(\overline{x}) \in K(A): A \overline{a} \mbox{ omits } \Phi \mbox{ for some (any) realization $\overline{a}$ of $q(\overline{x})$}\}$. By applying Lemma~\ref{amalg2} to each $p_n(\overline{x}_n)$ we get that $X$ is $\aleph_1$-comeager in $K(A)$. Hence by Lemma~\ref{amalg1} we can find $q(\overline{x}) \in X \cap K^*(A)$. Let $\overline{a}$ realize $q(\overline{x})$; then the $\mathbb{P}$-amalgamation problem $((A, \Phi), (A \overline{a}, \Phi), (A, \Phi \cup \{q(\overline{x})\}))$ has no solution.

\end{proof}

\end{document}